\documentclass[11pt]{amsart}

\usepackage[all]{xy}
\usepackage{graphics}
\usepackage{color}
\usepackage[T1]{fontenc}
\usepackage{parskip}
\usepackage[colorlinks]{hyperref}
\usepackage{enumitem}

\usepackage{tikz}
\usetikzlibrary{decorations.markings}

\usepackage[colorinlistoftodos,prependcaption]{todonotes}

\newcommand{\C}{\mathcal{C}}
\newcommand{\T}{\mathcal{T}}
\newcommand{\Z}{\mathbb{Z}}

\newcommand{\R}{\mathbb{R}}

\newcommand{\bfq}{$Q_{BF}(\OP{MCG}(\B\Sigma_g))$}

\newcommand{\B}[1]{{\mathbf #1}}

\newtheorem{theorem}{Theorem}[section]
\newtheorem{theorem*}{Theorem}
\newtheorem{lemma}[theorem]{Lemma}

\newtheorem*{question*}{Question}

\theoremstyle{definition}

\newtheorem*{remark*}{Remark}
\newtheorem*{remarks*}{Remarks}
\newtheorem*{corollary*}{Corollary}

\numberwithin{figure}{section}
\numberwithin{table}{section}
\numberwithin{equation}{section}

\def\B{\mathbf}

\newcommand{\OP}{\operatorname}

\begin{document}

\title[On the entropy norm on $\OP{Diff}(\B\Sigma_g)$]
{On the entropy norm on the group of 
diffeomorphisms of closed oriented surface}
\author{Michael Brandenbursky}
\address{Ben-Gurion University, Israel}
\email{brandens@math.bgu.ac.il}
\author{Arpan Kabiraj}
\address{Ben-Gurion University, Israel}
\email{kabiraj@post.bgu.ac.il}

\subjclass[2000]{Primary 53; Secondary 57}

\begin{abstract}
We prove that the entropy norm on the group of diffeomorphisms
of a closed orientable surface of positive genus is unbounded. 
\end{abstract}

\maketitle


\section{Introduction}\label{S:intro} 
Let $\B M$ be a smooth compact manifold with some fixed
Riemannian metric. Let $f\colon\B M\to\B M$ be a continuous function. 
Recall that the topological entropy of $f$ may be defined as follows.
Let $\B d$ be the metric on $\B M$ induced by some Riemannian metric. For $p\in\B N$ define a new metric
$\B d_{f,p}$ on $\B M$ by 
$$\B d_{f,p}(x,y)=\max_{0\leq i\leq p} \B d(f^i(x),f^i(y)).$$
Let $\B M_f(p,\epsilon)$ be the minimal number of $\epsilon$-balls in the $\B d_{f,p}$-metric that cover $\B M$.
The topological entropy $h(f)$ is defined by
$$h(f)=\lim_{\epsilon\to 0}\limsup_{p\to\infty}\frac{\log{\B M_f(p,\epsilon)}}{p},$$
where the base of $\log$ is two.
It turns out that $h(f)$ does not depend on the choice 
of Riemannian metric, see \cite{Bowen, Dinaburg}.

In this note we consider the case when $\B M$ is a closed oriented surface $\B\Sigma_g$ of genus $g$. Denote by
$\OP{Diff}(\B\Sigma_g)$ the group of orientation preserving diffeomorphisms of $\B\Sigma_g$. 
Let 
$$\OP{Ent}(\B\Sigma_g)\subset\OP{Diff}(\B\Sigma_g)$$ 
be the set of entropy-zero diffeomorphisms.
This set is conjugation invariant and it generates $\OP{Diff}(\B\Sigma_g)$, see Lemma \ref{L:gen}.  
In other words, a diffeomorphism of $\B\Sigma_g$ is a finite product of entropy-zero diffeomorphisms. 
One may ask for a minimal decomposition and this question
leads to the concept of the entropy norm defined by
$$
\|f\|_{\OP{Ent}}:=\min\{k\in\B N\,|\,f=h_1\cdots h_k,\,h_i\in\OP{Ent}(\B\Sigma_g)\}.
$$
It is the word norm associated with the generating set $\OP{Ent}(\B\Sigma_g)$.
This set is conjugation invariant, so is the entropy
norm. The associated bi-invariant metric is denoted by $\B d_{\OP{Ent}}$.
It follows from the work of Burago-Ivanov-Polterovich \cite{MR2509711}
and Tsuboi \cite{MR2523489,MR2874899} that for many manifolds all conjugation
invariant norms on $\OP{Diff}(\B M)$ are bounded. Hence the
entropy norm is bounded in those cases. In particular, it is bounded in case $g=0$.

Entropy metric may be defined in the same way on the group $\OP{Ham}(\B\Sigma_g)$ 
of Hamiltonian diffeomorphisms of $\B\Sigma_g$, and on groups $\OP{Diff}(\B\Sigma_g, \OP{area})$ 
and $\OP{Diff}_0(\B\Sigma_g, \OP{area})$. It is related to the autonomous metric
 \cite{MR3391653, MR3044593, BKS, 1405.7931, MR2104597}.
Recently, the first author in collaboration with Marcinkowski showed that the entropy metric is unbounded 
on groups: $\OP{Ham}(\B\Sigma_g)$, $\OP{Diff}_0(\B\Sigma_g, \OP{area})$ 
and on $\OP{Diff}(\B\Sigma_g, \OP{area})$, see \cite{BM}. On the other hand, it is not known, 
and seems to be a difficult problem, whether $\OP{Diff}_0(\B\Sigma_g)$ is unbounded in case $g>0$.
In this work we discuss the case of $\OP{Diff}(\B\Sigma_g)$ where $g>0$. Our main result is the following

\begin{theorem*}\label{T:unboundedness}
Let $\B\Sigma_g$ be a closed oriented Riemannian surface of positive genus.
Then the diameter of $(\OP{Diff}(\B\Sigma_g), \B d_{\OP{Ent}})$ is infinite.
\end{theorem*}

\textbf{Remarks.}
\begin{itemize}
\item The above theorem holds for non-sporadic surfaces with punctures. The proof is exactly the same.

\item In \cite{BM} the first author in collaboration with Marcinkowski showed that 
the diameter of $(\OP{Diff}(\B\Sigma_g, \OP{area}), \B d_{\OP{Ent}})$ is infinite. 
Our proof of Theorem \ref{T:unboundedness}, which is simpler than the one given in \cite{BM}, is applicable to the case 
of $\OP{Diff}(\B\Sigma_g, \OP{area})$.

\item It would be interesting to know whether the entropy metric, or the autonomous metric are unbounded 
on $\OP{Diff}_0(\B\Sigma_g)$ in case $g>0$.
\end{itemize} 

\subsection* {Acknowledgments.}
First author was partially supported by Leverhulme Trust Grant RPG-2017-159.
Second author was partially supported by the center of advanced studies at
Ben Gurion University, by GIF-Young Grant number I-2419-304.6/2016, 
by ISF Grant number 2095/15 and by DST-INSPIRE, India. 
He wishes to express his gratitude to the center for the support and excellent working conditions.

\section{Preliminaries}\label{S:preliminaries} 
Let us start with the following
\begin{lemma}\label{L:gen}
Let $\B\Sigma_g$ be a closed oriented surface of genus $g$. Then $\OP{Diff}(\B\Sigma_g)$ is generated
by the set $\OP{Ent}(\B\Sigma_g)$ of entropy zero diffeomorphisms.
\end{lemma}
\begin{proof}
The group $\OP{Diff}_0(\B\Sigma_g)$ is simple and hence is generated by entropy zero diffeomorphisms. It is enough to prove
the lemma in case $g>0$ since $\OP{Diff}(\B\Sigma_0)=\OP{Diff}_0(\B\Sigma_0)$.
In addition, Dehn twists have entropy zero and they generate $\OP{Diff}(\B\Sigma_g)/\OP{Diff}_0(\B\Sigma_g)$ in
case $g>1$. Hence in this case $\OP{Diff}(\B\Sigma_g)$ is generated by entropy zero diffeomorphisms.
In case $g=1$ we have that  
$$\OP{Diff}(\B\Sigma_1)/\OP{Diff}_0(\B\Sigma_1)\cong  \OP{SL}_2(\Z),$$ 
which in turn is generated by two matrices of finite order. Hence in this case $\OP{Diff}(\B\Sigma_g)$ 
is also generated by entropy zero diffeomorphisms.
\end{proof}

Let $\B\Sigma_g$ be a closed oriented surface of genus $g>1$.
\subsection{Translation length in Teichm{\"u}ller space} 
We denote the Teichm{\"u}ller space associated to $\B\Sigma_g$ by $\T(\B\Sigma_g)$. We equip $\T(\B\Sigma_g)$ with the Teichm{\"u}ller metric $\B d_{\T}$. 
Let $\OP{MCG}(\B\Sigma_g)$ be the mapping class group of $\B\Sigma_g$, i.e., $\OP{MCG}(\B\Sigma_g):=\OP{Diff}(\B\Sigma_g)/\OP{Diff}_0(\B\Sigma_g)$.
Note that it acts naturally on $\T(\B\Sigma_g)$. Let $[f]\in \OP{MCG}(\B\Sigma_g)$. 
The \emph{translation length} of $[f]$ in $\T(\B\Sigma_g)$ is defined by  
$$\tau_\T([f])=\underset{n\to\infty}{\lim}\frac{\B d_\T([f]^n(X),X)}{n}$$
where  $X\in\T(\B\Sigma_g)$. It is independent of the choice of $X$.

Let $[f]\in \OP{MCG}(\B\Sigma_g)$ be a pseudo-Anosov element with dilatation $\lambda_{[f]}$. 
According to Bers \cite{Bers} proof of Thurston's classification theorem of elements of mapping class group we have:
\begin{itemize}
\item there exists $X\in \T(\B\Sigma_g)$ such that $\tau_{\T}([f])=\B d_\T([f](X),X)$,\\
\item $\tau_{\T}([f])=\log(\lambda_{[f]})$.
\end{itemize}  

\subsection{Translation length in curve complex}
Given a surface $\B\Sigma_g$, we associate to it a simplicial complex as follows: its vertices are free homotopy classes of essential simple closed curves; 
a collection of $n+1$ vertices form an $n$-simplex whenever it can be realized by pairwise disjoint closed curves in $\B\Sigma_g$.
This complex is called the \emph{curve complex} of $\B\Sigma_g$ and is denoted by $\C(\B\Sigma_g)$. 
It is known that $\C(\B\Sigma_g)$ is connected. We consider the path metric on the $1$-skeleton of $\C(\B\Sigma_g)$ and denote it by $\B d_{\C}$.

Mapping class group $\OP{MCG}(\B\Sigma_g)$ acts by isometry on $\C(\B\Sigma_g)$. 
Given a mapping class $[f]\in\OP{MCG}(\B\Sigma_g)$, the \emph{translation length} of $[f]$ in $\C(\B\Sigma_g)$ is defined by
$$\tau_\C([f])=\underset{n\to\infty}{\lim}\frac{\B d_\C([f]^n(\alpha),\alpha)}{n}$$ 
where $\alpha$ is a vertex in $\C(\B\Sigma_g)$. 
The translation length is independent of $\alpha$ and is non-zero 
if and only if $[f]$ is a pseudo-Anosov mapping class \cite{MR1714338}. 

\subsection{Bestvina-Fujiwara quasimorphisms}
Let $G$ be a group. Recall that a function $\psi:G\rightarrow\R$ is 
called a quasimorphism if there exists $D>0$ such that
$$|\psi(ab)-\psi(a)-\psi(b)|<D$$ 
for all $a,b\in G$. A quasimorphism $\psi$ is called homogeneous 
if $\psi(a^n)=n\psi(a)$ for all $n\in \Z$ and all $a\in G$. 
Given a quasimorphism $\psi$ we can always 
construct a homogeneous quasimorphism $\overline{\psi}$ by setting 
$$\overline{\psi}(a):=\underset{p\to\infty}{\lim}\frac{\psi(a^p)}{p}$$

In \cite{MR1914565}, Bestvina and Fujiwara constructed infinitely many  
homogeneous quasimorphisms on $\OP{MCG}(\B\Sigma_g)$. Let us recall their construction.

Let $w$ be a finite oriented path in $\C(\B\Sigma_g)$. Denote the length of a path $\omega$ by $|\omega|$. 
For any finite path $\sigma$ in $\C(\B\Sigma_g)$, we define 
$$|\sigma|_\omega:= \{\text{the number of non-overlapping copies of $\omega$ in $\sigma$}\}.$$
Fix a positive integer $W<|\omega|$. Given any two vertices $\alpha,\beta\in\C(\B\Sigma_g)$, define 
$$c_{\omega,W}(\alpha,\beta)=\B d_{\C}(\alpha,\beta)-\inf (|\sigma|-W|\sigma|_\omega),$$ 
where the infimum is taken over all paths $\sigma$ between $\alpha$ and $\beta$. 

It turns out that the function $\psi_\omega:\OP{MCG}(\B\Sigma_g)\rightarrow \R$ defined by 
$$\psi_{\omega}([f])=c_{\omega,W}(\alpha,[f](\alpha))-c_{\omega^{-1},W}(\alpha,[f](\alpha)),$$  
where $\alpha$ is a vertex of $\C(\B\Sigma_g)$, is a quasimorphism  \cite{MR1914565}. 
The induced homogeneous quasimorphism is denoted by $\overline{\psi}_\omega$. 
We denote by  \bfq the space of homogeneous quasimorphisms on $\OP{MCG}(\B\Sigma_g)$ which is spanned by
Bestvina-Fujiwara quasimorphisms. In \cite{MR1914565} it is proved that  \bfq
is infinite dimensional whenever $\B\Sigma_g$ is a non-sporadic surface.

\section{Proof of the main result} Let us start with the following well-known 
\begin{lemma}\label{L:unb}
Let $G$ be a group generated by set $S$ and let $\psi:G\to\mathbb{R}$ be a non-trivial 
homogeneous quasimorphism which vanishes on $S$. Then the induced word norm $\|\cdot\|_S$
is unbounded.
\end{lemma}
For the reader convenience we present its proof.

\begin{proof}
Let $g\in G$ such that $\psi(g)\neq 0$. Then $g=s_1\cdot\ldots\cdot s_{\|g\|_S}$. It follows that
$|\psi(g)|\leq \|g\|_S D_\psi$. Hence for each $n$ we get $\|g^n\|_S\geq n|\psi(g)|/D_\psi$ and the
proof follows.
\end{proof}

Now we prove Theorem \ref{T:unboundedness}.

\textbf{Case 1.} Let $g=1$ and denote $\B T:=\B\Sigma_1$. 
Let us consider homomorphism $F:\OP{Diff}(\B T)\to \OP{SL}_2(\Z)$
induced by the action of a diffeomorphism on the first homology $H_1(\B T,\Z)$. 
It is known that $F$ is surjective  (see \cite[Theorem 2.5]{FM}). 
By \cite[Theorem 1]{MR0650661}, $\log(\OP{spec}(f))\leq h(f)$ where $\OP{spec}(f) $
is the  modulus of the largest eigenvalue of $F(f)$. 
Therefore if $f$ has entropy zero then the modulus of the eigenvalues of $F(f)$ is at most one.

There are three types of elements in $\OP{SL}_2(\Z)$: \emph{periodic} (trace$<$2), 
\emph{parabolic} (trace=2) and \emph{hyperbolic} (trace$>$2). 
Therefore if $F(f)$ is hyperbolic then $\OP{spec}(f)>1$ and hence $h(f)>0$. 
Hence if $f$ is an entropy zero diffeomorphism, then $F(f)$ is either parabolic or periodic.

The value of any homogeneous quasimorphism on a periodic element is zero. 
It follows from the work of Polterovich and Rudnick \cite[Proposition 3]{MR2054053} that there exists a non-trivial 
homogeneous quasimorphism on $\OP{SL}_2(\Z)$ which vanishes on parabolic elements. 
Therefore there exists a non-trivial homogeneous quasimorphim on $\OP{Diff}(\B T)$ whose 
restriction on entropy-zero diffeomorphisms is zero. Hence by Lemma \ref{L:unb} 
the entropy norm on $\OP{Diff}(\B T)$ is unbounded.

\textbf{Case 2.} Let $g>1$. Given a homeomorphism $f$ of a surface $\B\Sigma_g$ define 
$$H(f)=\inf\{h(f'):f' \text{ is isotopic to }f\}$$ 
The topological entropy of $[f]\in\OP{MCG}(\B\Sigma_g)$ is defined to be $H(f)$.

\begin{lemma}\label{L:main}
Each quasimorphism in \bfq is Lipschitz with respect to the topological entropy. 
\end{lemma}
\begin{proof}
Let $\psi\in $\bfq. If $[f]$ is reducible then $\psi([f])=0$ for all $\psi\in$ \bfq. 
Therefore it is enough to consider only pseudo-Anosov elements of $\OP{MCG}(\B\Sigma_g)$.
Since ${\psi}\in $\bfq, then $\psi=\sum _i^k a_i\overline{\psi}_{w_i}$, where 
$a_1, \ldots, a_k \in\mathbb{R}$ and $w_1, \ldots, w_k$ are some paths in $\C(S)$. 
It follows from the definition of $\overline{\psi}_{w_i}$ that $\overline{\psi}_{w_i}([f])\leq\tau_{{\C}}([f])$ 
for each $[f]\in\OP{MCG}(\B\Sigma_g)$ and each $i\in\{1, \ldots ,k\}$. Therefore we have 
$$|\psi([f])|\leq (\sum _{i=1}^k |a_i|)\tau_{\C}([f]).$$
By setting $C_\psi:=\sum _{i=1}^k |a_i|$ we get $|\psi ([f])|\leq C_{\psi}\tau_{\C}([f])$.

Let $sys:\T(\B\Sigma_g)\to\C(\B\Sigma_g)$ be the systole function, 
i.e., $X\in\T(\B\Sigma_g)$ goes to a vertex in $\C(\B\Sigma_g)$ which corresponds to a simple closed curve of minimal length in $X$. 
By \cite{MR1714338} there exist $K,C>0$ such that for all $X,Y\in \T(\B\Sigma_g)$
$$\B d_{\C}(sys(X),sys(Y))\leq K\B d_{\T}(X,Y)+C.$$
It is immediate that $[f]^n(sys(X))=sys([f]^n(X))$ for every $[f]\in\OP{MCG}(\B\Sigma_g)$.

Let $[f]\in\OP{MCG}(\B\Sigma_g)$ be a pseudo-Anosov element with dilatation $\lambda_{[f]}$. 
It follows from Bers \cite{Bers} proof of Thurston's theorem that $\tau_{\T}([f])=\log\lambda_{[f]}$. 
Therefore
\begin{align*}
\frac{\tau_{\C}([f])}{\tau_{\T}([f])} & = \underset{n \to \infty}{\lim}\frac{\frac{\B d_{\C}(sys(X),[f]^n(sys(X)))}{n}}{\frac{\B d_{\T}(X,[f]^n(X))}{n}}\\
 & = \underset{n \to \infty}{\lim} \frac{\frac{\B d_{\C}(sys(X),sys([f]^n(X)))}{n}}{\frac{\B d_{\T}(X,[f]^n(X))}{n}}\\
 & \leq\underset{n \to \infty}{\lim}\frac{K\B d_{\T}(X,[f]^n(X))+C}{\B d_{\T}(X,[f]^n(X))}= K
\end{align*} 
Thus
$$\tau_{\C}([f])\leq K{\tau_{\T}([f])}.$$ 
It follows that for each $\psi\in$ \bfq we have
$$
|\psi([f])|\leq C_{\psi}\tau_{\C}([f])\leq C_{\psi}K\tau_{\T}([f])=C_{\psi}K\log\lambda_{[f]}.
$$
By Thuston's result \cite[Proposition 10.13]{MR3053012}, $\log\lambda_{[f]}=H(f)$. Hence
$$|\psi([f])|\leq C_{\psi}KH(f)$$
and the proof of the lemma follows.
\end{proof}
Let $\Pi:\OP{Diff}(\B\Sigma_g)\to\OP{MCG}(\B\Sigma_g)$ be the quotient map and let $\psi\in $\bfq.
It follows from the proof of Lemma \ref{L:main} that for each $f\in\OP{Diff}(\B\Sigma_g)$ we have
$$|\psi\Pi(f)|\leq C_{\psi}KH(f)\leq C_{\psi}Kh(f).$$
Hence for each non-trivial $\psi\in $\bfq  the homogeneous quasimorphism 
$$\psi\Pi:\OP{Diff}(\B\Sigma_g)\to\mathbb{R}$$
is non-trivial and Lipschitz with respect to the topological entropy. 
It follows that it vanishes on the set of entropy-zero diffeomorphisms.
Hence by Lemma \ref{L:unb}  the entropy norm on $\OP{Diff}(\B\Sigma_g)$ is unbounded.
\qed
\bibliography{bibliography}
\bibliographystyle{plain}

\end{document}